\documentclass[onecolumn,11pt,journal]{article}
\usepackage{cite,newalg,epsfig}
\usepackage{subfigure,amsmath,times,amssymb,psfrag,graphicx,color,setspace}
\usepackage{multirow}

\usepackage{epsfig}
\usepackage{latexsym}
\usepackage{url}
\usepackage[hypertexnames=false]{hyperref}
\usepackage{ifthen}
\usepackage{subfigure}\usepackage{rotating}

\def\blfootnote{\xdef\@thefnmark{}\@footnotetext}
\long\def\symbolfootnote[#1]#2{\begingroup%
\def\thefootnote{\fnsymbol{footnote}}\footnote[#1]{#2}\endgroup}

\newcommand{\Sscr}{{\cal S}}

\newtheorem{theorem}{Theorem}
\newtheorem{proposition}{Proposition}

\newtheorem{lemma}{Lemma}
\newtheorem{assumption}{Assumption}

\newtheorem{example}{Example}

\newenvironment{proof}{\noindent {\bf Proof: }}{\hfill \qed \par}

\DeclareMathOperator{\argmax}{argmax}

\newcommand{\eps}{{\epsilon}}
\newcommand{\Ascr}{{\mathcal A}}

\newcommand{\Nscr}{{\mathcal N}}
\newcommand{\Jscr}{{\mathcal J}}

\newcommand{\1}{{\bf 1}}
\newcommand{\qed}{ $\blacksquare$}

\usepackage{setspace}
\begin{document}
\thispagestyle{empty}
\setcounter{page}{0}
\title{Efficiency-Loss of Greedy Schedules in Non-Preemptive Processing of Jobs with Decaying Value}

\author{\begin{tabular}[t]{c@{\extracolsep{8em}}c}
Carri W. Chan  & Nicholas Bambos  \\
         & Electrical Engineering Department and \\ 
        Electrical Engineering Department& Management Science \& Engineering Department \\
        Stanford University & Stanford University \\
        Stanford, CA 94305 & Stanford, CA 94305 \\
        cwchan@stanford.edu & bambos@stanford.edu
\end{tabular}}
\date{}
\maketitle
\thispagestyle{empty}
\pagebreak

\pagenumbering{arabic}
\doublespacing
\begin{center}
{\LARGE Efficiency-Loss of Greedy Schedules in Non-Preemptive Processing of Jobs with Decaying Value}
\end{center}

\begin{abstract}

We consider the problem of dynamically scheduling $J$ jobs on $N$ processors for {\em non-preemptive} execution where the value of each job (or the reward garnered upon completion) {\em decays} over time. All jobs are initially available in a buffer and the distribution of their service times are known. When a processor becomes available, one must determine which free job to schedule so as to maximize the total expected reward accrued for the completion of all jobs. Such problems arise in diverse application areas, e.g. scheduling of patients for medical procedures, supply chains of perishable goods, packet scheduling for delay-sensitive communication network traffic, etc. Computation of optimal schedules is generally intractable, while online low-complexity schedules are often essential in practice.

It is shown that the simple greedy/myopic schedule provably achieves performance within a factor $2+\frac{ E[\max_j\sigma_j]}{\min_j E[\sigma_j]}$ from optimal. This bound can be improved to a factor of $2$ when the service times are identically distributed.  Various aspects of the greedy schedule are examined and it is  demonstrated to perform quite close to optimal in some practical situations despite the fact that it ignores reward-decay deeper in time.
\end{abstract}

\section{Introduction}

Consider a queueing/scheduling system (as in Fig. \ref{fig:sys}), where a finite number $J$ jobs wait in a buffer, each to be processed by one of $N$ servers/processors.  Time is slotted. The service/processing requirement, $\sigma_j$, of each job $j$ is random and its distribution, $f_j(\sigma_j)$, is known. All processors operate at service rate $1$; hence, the service time for each job is invariant to the processor which it assigned. Service is {\em non-preemptive} (job service cannot be interrupted mid-processing to be resumed later or discontinued). The completion of job $j$ in time slot $t$ garners a reward $w_j(t)\geq 0$, which decays with time (i.e. $w_j(t)$ is non-increasing in $t$).  The goal is to schedule the jobs on the processors so as to maximize the aggregate reward accrued when all jobs complete execution.

As will become clear below, a key complicating factor is that the job service is {\em non-preemptive}, inducing a `combinatorial twist' on the problem. Under preemptive processing, the latter would wash away and the problem would become much simpler. Another complicating factor is the fact that the rewards/values $w_j(t)$ decay over time in a general way; special cases might be significantly easier to handle (though still not necessarily easy).
A third complicating factor is the general distributions of the stochastic job processing times $\sigma_j$ (even though these are independent across different jobs); for special distributions the problem can become significantly simpler (and the results tighter). We aim to address the problem in the most general setting arising in a variety of applications (see below), which may actually require online (real-time) schedule implementation. In that case, since the complexity of computing the optimal job schedule is prohibitive, one seeks simple and practical schedules (implementable online), which have performance within provable bounds from optimal. In this paper, we focus on a \emph{greedy/myopic} schedule defined below and study its efficiency. We discuss these factors below in conjunction with prior work and a variety of applications.

\begin{figure}
\begin{center}
\psfrag{jobs}{jobs $ =
\{1,2,\dots,j,\dots,J-1,J\}$}\psfrag{m1j}{$\sigma_j$}\psfrag{m2j}{$\sigma_j$}\psfrag{mn1j}{$\sigma_j$}\psfrag{mnj}{$\sigma_j$}
\includegraphics[scale=.8]{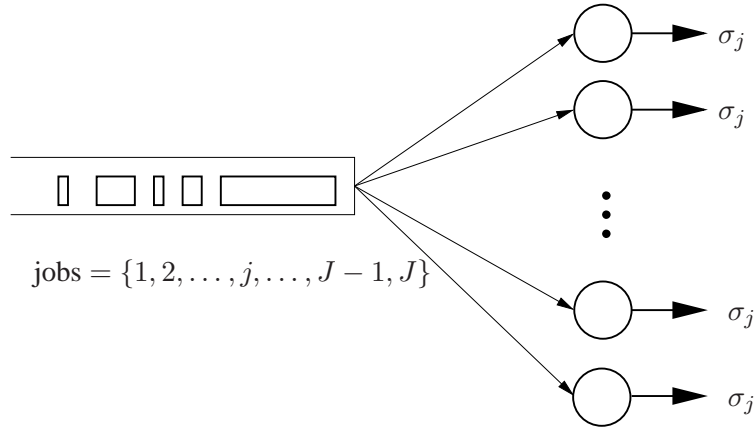}
\caption{System Diagram: $J$ jobs wait to be processed on one of $n$ machines.  The processing time of each job is independent of the processor and other jobs.}\label{fig:sys}
\end{center}
\end{figure}

\subsection{Applications}

There are diverse applications where job completion rewards decay over time.  For example, such is the case with patient scheduling in health-care systems. Delays in treatment often lead to deterioration of patient health (see, for instance,\cite{mcquillan}) which may result in reduction of the eventual treatment impact; this is obviously the case with various medical procedures, operations, etc. Indeed, a number of studies have demonstrated that delayed treatment results in increased patient mortality \cite{chan,luca,buist,bellomo,sharek_jama}. Moreover, in a related study \cite{poon_arch_internmed}, over 60\% of physicians reported dissatisfaction with delays in viewing test results, which subsequently led to delays in treatment. It is likely that increased mortality is primarily induced via deterioration of patient health condition and resulting reduction of benefit from eventual treatment.  This is how the effect of treatment delay is modeled in this paper.

On the other hand, in information technology, reward decay occurs in various situations--for example, in multimedia packet scheduling for transmission over wireless links. Each packet corresponds to a job which is completed once the packet is successfully received at the receiver; until then, it is repeatedly transmitted (non-preemptive processing). Transmission time until successful reception is random, due both to random packet sizes and randomly varying wireless channel quality. In the simplest case, video packets have a single deadline and reward is only received if the packet is received prior to its deadline expiration.  In more advanced schemes, multiple deadlines are considered (decreasing, piecewise-constant reward decay function), reflecting coding interdependencies across packets. Indeed, even if a packet misses its initial deadline, it could improve the quality of the received and reconstructed video because other packets which depend on it may still be able to meet their deadline \cite{wee_icme02}.

As with multimedia packet scheduling above and similar situations of task scheduling in parallel computing systems, we can consider jobs that contain interdependencies within our model.  The completion of a single job $j$ garners reward $r_j$.  However, other jobs  may rely on that one too, either because they cannot begin processing until that is completed (due to data-passing, precedence constraints, etc.) or their processing accuracy/quality depends on output from that job (e.g. decoding dependencies). Therefore,  the `effective' reward generated is actually $w_j(t) = r_j - f(t)$, where the increasing function $f(t)$ reflects the detrimental effect that completing job $j$ after delay $t$ has on other jobs depending on it.  In fact, our formulation allows for the case where even $r_j$ is a decaying function in time.

A third application area where job completion rewards may decay over time is in the case of perishable items, like food, medicine, etc.  For example, the quality of food items (milk, eggs, etc.) decays with time.  The scheduling problem is when to release these items for sale given varying transportation times (from storage to shelf) and the decaying reward $R(t)$.  It is also possible to have a cost $s$ for each time slot the item remains in storage so that the effective reward of an item once it is released for sale is $C(t) = R(t) - st$.

\subsection{Literature Review}

When rewards do not decay over time but stay constant, job scheduling problems may be cast in the framework of `multiarmed bandit' problems \cite{gittins, walrand}.  Furthermore, optimal policies for certain `well-behaved' decaying reward functions (such as linear and exponential) have been developed (see \cite{gittins, walrand} and related works).  Unfortunately, under general decaying rewards, solving for the optimal schedule becomes very difficult.

There has been related work on delay-sensitive scheduling in networking.  In the case of broadcast scheduling in computer networks, jobs correspond to requests for pages (files).  Due to the broadcast nature of a wireless channel, multiple requests can be satisfied with the transmission of a single page. In \cite{kim04}, a greedy algorithm is shown to be a $2$-approximation for throughput maximization of broadcast scheduling in the case of {\em equal sized pages}.  In a similar scenario, an online {\em preemptive} algorithm is shown to be $\Omega(\sqrt{n})$ competitive where $n$ is the number of pages that can be requested \cite{lipton94}. Our work differs from this prior work in that we allow for 1) arbitrary decaying rewards, rather restricting to step functions when the deadline expires, 2) jobs are non-preemptive and have varied lengths (and all jobs are available at time $0$).

A substantial body of work has focused on scheduling for perishable products (see \cite{nahmias82} for a review).  The focus is on finding an optimal ordering policy given the lifetime and demand of the perishable items.  In \cite{naso07}, the authors study how to maximize utility garnered by delivering perishable goods, such as ready-mixed concrete, and minimize costs subject to stochasticity in transportation times.  The authors formulate a mathematical program to solve the problem and propose heuristic algorithms for use in practice.  Interestingly, the perishable items in this case have a fixed lifetime, after which they are rendered useless (deadline).  Our formulation here allows for general decay.

In \cite{dalal05}, the authors look at how to schedule an M/M/1 queue where rewards decay exponentially dependent on each job's sojourn time due to the `impatient' nature of the users.  A greedy policy is shown to be optimal in the case of identical decay rates of these impatient users.  Our scheduling problem is closely related to a number of instances of the Multiarmed Bandit Problem.  When rewards exhibit `well-behaved' decay, (identical rates, constant rates, etc.) it is possible to find optimal, or near-optimal algorithms \cite{pinedo,schuurman99,walrand,gittins}.  This is not always the case for arbitrary decay.

In a problem similar to the one we study in this paper, a greedy algorithm is shown to be a $2$-approximation when job completions generate rewards according to general decaying reward functions \cite{chan08}.  The main distinction between this work and ours is that the previous work allows for job preemption while we consider the case that once a job is scheduled it occupies the machine until it completes.  This constraint adds an extra layer of complexity.

Indeed, non-preemption makes the scheduling problem we study substantially more difficult.  Non-preemptive interval scheduling is studied in \cite{goldman00,lipton94} among others.  Jobs can either be scheduled during their specified interval or rejected.  The end of the interval corresponds to the deadline of the corresponding job.  If $\Delta$ is the ratio of the large job size to the smallest job size, then an online algorithm cannot be better than $O(\log \Delta)$.  Our work differs from this prior work because we consider arbitrary decay of rewards and assume all jobs are available at time $0$. The decaying reward functions make this a more general and difficult scheduling problem.
However, our result also relies on $\Delta$, the ratio between largest and smallest jobs.

Still, there are instances where optimal schedules can be found for arbitrary decaying rewards. In a parallel scenario to ours, jobs can be scheduled, non-preemptively,  multiple times.  For this problem, the reward function for completing a particular job decays with the number of times that job has been completed. In this case, a greedy policy is optimal for arbitrary decaying rewards \cite{walrand}.  This problem is parallel to ours in that it allows for arbitrary decaying rewards.  However, the decay does not depend on the completion time of the job, but rather on the number of times that job has been completed.  In our case, each job is only processed a single time.

Relating back to our scenario where the rewards decay with time, it is again the case that for `well-behaved' decaying functions (linear and exponential), policies based on an index rule are optimal \cite{gittins,walrand}.  The policy we propose in this paper is also an index rule.  In fact, the proposed policy  is very closely related to the `c-$\mu$'-type scheduling rules (see, for instance \cite{VanMieghem,walrand}) where the objective is to minimize cost (rather than maximize rewards) when costs are linearly or concavely increasing.  One of the main distinctions between our work and this is that we consider multiple servers.  Unfortunately, the optimality of the `c-$\mu$' rule does not extend to this case.  Furthermore, linear/concave decaying rewards are just single instances of our more general formulation of decaying rewards.  It is also important to recognize that many of the results of this prior work are in heavy-traffic regimes where a lot of the fine-grained optimization required in non-heavy-traffic is washed out.

\subsection{Summary of Results}

In this paper, we study the efficacy of a greedy scheduling algorithm for non-preemptive jobs whose rewards decay arbitrarily with time.  There are a number of applications which exhibit such behavior such as patient scheduling in hospitals, packet scheduling in multimedia communication systems, and supply chain management for perishable goods.  It is shown that finding an optimal scheduling policy for such systems is NP-hard.  As such, finding simple heuristics is highly desirable.  We show that a greedy algorithm is guaranteed to be within a factor of $\Delta+2$ of optimal where $\Delta$ is the ratio of the largest job completion time to the smallest.  This bound is improved in some special cases.  Via numerical studies, we see that, in practice, the greedy policy is likely to perform much closer to optimal which suggests it is a reasonable heuristic for practical deployment.  To the best of our knowledge this is the first look at non-preemptive scheduling of jobs with arbitrary decaying rewards.

The rest of the paper is structured as follows.  In Section \ref{sec:model} we formally introduce the scheduling model we will study.  In Section \ref{sec:greedy} we propose and study the performance of a greedy scheduling policy. The main result, which is a bound on the loss of efficiency due to greedy scheduling, is given in Section \ref{ssec:bound}.  In Section \ref{sec:special}, we examine some special cases where this bound can be improved.  In Section \ref{sec:perfeval}, we do some performance evaluation of the greedy policy via a simulation study.  Finally, we conclude in Section \ref{sec:conc}.

\section{Model Formulation}\label{sec:model}

Consider a set of $J$ jobs, indexed by $j \in \Jscr =\{1,2,\dots,J\}$, and $N$ processors/servers, indexed by $n \in \Nscr = \{1,2,\dots, N\}$.  Each job $j\in \Jscr$ has a random processing requirement $\sigma_j$ and can be processed by any processor $n\in \Nscr$. All processors have service rate 1 and each one can process a single job at a time. Service is {\em non-preemptive} in the sense that once a processor starts executing a job it cannot stop until completion. Time is slotted and indexed by $t\in\{1,2,3,...\}$.  We denote the distribution of the service times by $f_j(\sigma_j)$.
\begin{assumption}
\label{as:sigma}
The random job processing times $\sigma_j,j\in\cal J$ are 1) statistically independent with $P(\sigma_j < \infty) = 1$ and 2) their distributions, $f_j(\sigma_j)$, do not depend on time.
\end{assumption}
However, the jobs processing times are {\em not} necessarily identically distributed.

Let $b_j(t)$ be the residual service time of job $j$ at time $t$. Initially, $b_j(0)=\sigma_j$, for each $j\in\cal J$. The backlog state of the system at time $t$ is the vector
\begin{equation}
{\bf b}(t)= \left(b_1(t), b_2(t),..., b_j(t), ..., b_J(t)\right).
\end{equation}
It evolves from initial state ${\bf b}(0)= (\sigma_1,\sigma_2,...,\sigma_j,...,\sigma_J)$ to final state $\vec \sigma(T)=(0,0,...,0,...,0)$ by assigning processors to process the jobs \emph{non-preemptively}, until all jobs have finished execution at some (random) time $T$.
Note that for each job $j\in\cal J$, $b_j(t)=\sigma_j$ implies that $j$ has not started processing by $t$ (has not been scheduled before $t$), while $b_j(t)=0$ implies that the job finished execution before (or at) time $t$. Indeed, if job $j$ starts execution at time slot $t_j$ and finishes at the beginning of time slot $T_j$ then $\sigma_j=T_j-t_j$ and
\begin{equation}
b_j(t) = \left\{
\begin{array}{lll}
\sigma_j, & t < t_j \\
\sigma_j-[t-t_j], & t = t_j, t_j+1, ..., T_j-1\\
0, & t \geq T_j
\end{array} \right.
\end{equation}
As discussed later, the start times $t_j$ are chosen by the scheduling policy, while the end times $T_j$ are then determined by the fact that scheduling is non-preemptive so that  $T_j=t_j+ \sigma_j$.

The job service times are random and their true values are not observable ex ante or known a priori; they can only be seen ex post, after a job has completed processing. However, the values $x_j(t)$ tracking which jobs are completed at each time $t$
\begin{equation}
x_j(t) = \left\{
\begin{array}{ll}
1, & \mbox{if job $j$ has not completed processing at time slot $t$} \\
0, & \mbox{if job $j$ has completed by time slot $t$}
\end{array}
\right.
\end{equation}
are directly observable for each job $j\in\cal J$. We work below with the {\em observable} `backlog state'
\begin{equation}
{\bf x}(t)=\left(x_1(t),x_2(t),...,x_j(t),...,x_J(t)\right)
\end{equation}
in $\{0,1\}^J$ which tracks which jobs are completed and which are still waiting to complete processing at time $t$.

To fully specify the state of a job, we define $y_j(t)$ as the time slot $t' < t$ in which job $j$ begins processing. Specifically,
\begin{equation}
y_j(t) = \left\{
\begin{array}{ll}
t', & \mbox{if job $j$ began procesing in time slot $t' < t$} \\
\emptyset, & \mbox{if job $j$ has not begun processing prior to time slot $t$ (necessarily, $x_j(t) = 1$)}
\end{array}
\right.
\end{equation}
Hence, any job with $y_j(t) = \emptyset$ (where $\emptyset$ is some null symbol)  has not yet begun processing and is free to be scheduled.   If $x_j(t) = 1$, then job $j$ has not completed and it is still being processed due to the non-preemptive nature of the service discipline.  Once a job is scheduled in time slot $t_j$, then $y_j(t) = t_j$ for all $t > t_j$.  The service state is then,
\begin{equation}
{\bf y}(t)=(y_1(t),y_2(t),...,y_j(t),...,y_J(t))
\end{equation}
in $\{\{0,1,...,t-1\} \cup \emptyset \}^J$ and tracks when (and if) each job began processing.   In time slot $t$, one can calculate the distribution for the remaining service time $b_j(t)$ given the distribution of $\sigma_j$ based on when (if) the job has started processing and whether it has completed.  Only the distribution of $b_j(t)$ is known as the job service time is only observable once the job completes processing. Therefore,  $x_j$ and $y_j$ can be jointly leveraged to compute the distribution of the residual service time of job $j$.

We next define the state $z_n(t)$ of processor $n\in\cal N$ which tracks which job it is assigned to process in time slot $t$. Specifically,
\begin{equation}
z_n(t)= \left\{
\begin{array}{ll}
j, & \mbox{if processor $n$ is still executing job $j\in\cal J$ at the beginning of time slot $t$} \\
0, & \mbox{if processor $n$ is free at the beginning of time slot $t$, hence, available for allocation}
\end{array}
\right.
\end{equation}
and the processor state is
\begin{equation}
{\bf z}(t)=(z_1(t),z_2(t),...,z_n(t),...,z_N(t))
\end{equation}
in $\{0,1,...,j,...J\}^N$ and tracks the free vs. allocated processors at the beginning of time slot $t$.

At the beginning of each time slot $t$, each job $j$ with $y_j(t) = \emptyset$ (not yet started) can be scheduled on (matched with) a processor $n$ with $z_n(t) = 0$ (free) to start execution. The {\em observable} state of the system at the beginning of time slot $t$ is
\begin{equation}\label{eq:state}
s_t = \left( {\bf x}(t),{\bf y}(t),{\bf z}(t)  \right)
\end{equation}
Recall that from ${\bf x}(t)$ and ${\bf y}(t)$ we can determine the distribution of the remaining service time ${\bf b}(t)$.  So the global state (\ref{eq:state}) yields the distribution for the remaining backlog and also tracks the processor state.
The state space $\Sscr$ is the set of all states the system may attain throughout its evolution.  We denote by $\mathbf{x}(s)$ the projection of the state onto the $\mathbf{x}$-coordinate.  We similarly apply notation for $\mathbf{y}(s)$ and $\mathbf{z}(s)$.

Given the free jobs and processors at state $s_t$, we denote by $\Ascr(s_t)$ the set of job-processor matchings (schedules) that can be selected, i.e. they are feasible, at the beginning of time slot $t$. These matchings are in addition to those already in place for jobs which are in mid-processing due to the non-preemptive nature of execution. Note that at each time $t$, for any feasible job-processor matching $A \in \Ascr(s_t)$ we have that $(j,n) \in A$ implies $x_j (t)= 1$, $y_j(t) = \emptyset$ and $z_n(t) = 0$, meaning processor $n$ is free and job $j$ has not started processing.  Also, only one free job can be matched to each free processor and vice-versa (hence, $(j,n), (k,m) \in A$ with $(j,n) \not = (k,m)$ implies $j \not = k$ and $n \not = m$).  Despite the fact that $\Ascr(s_t)$ is clearly a function of $s_t$, we may occasionally suppress $s_t$ for notational simplicity.

The completion of job $j$ by the end of time slot $t$ garners non-negative reward $w_j(t)$. We assume the reward decays over time, as follows.
\begin{assumption}
\label{as:reward}
For each job $j\in\cal J$, the reward function  $w_j(t) \geq 0$ decays  over time; that is, it is non-increasing in $t$ (may be piece-wise constant).
\end{assumption}
This immediately accounts for raw deadlines by setting $w_j(t)={\bf 1}_{ \{t\leq d_j\} }$ when $d_j$ is the deadline of job $j$.

Recall that if job $j$ is scheduled on processor $n$ at the beginning of time slot $t$, it will finish by the beginning of time slot $t+\sigma_j$. Therefore, the reward `locked' at the beginning of time slot $t$, given that a job-processor match $A\in\Ascr(s_t)$ is chosen to be used in this slot, is simply
\begin{equation}
\begin{split}
R_t(s,A) = \sum_{(j,n) \in A}w_j(t + \sigma_j)
\end{split}
\end{equation}

It is desirable to design a control (scheduling, matching) policy choosing at each $t$ a job-processor matching in $\Ascr(s_t)$ to maximize the total expected reward accrued until all jobs have been executed. Since at time $t$ the realization of $\sigma_j$  is unknown for each job $j$ that has not completed by $t$, \emph{any} control policy is a-priori unaware of the exact reward accrued from a particular action at $t$. Only the statistics of this reward are known. Specifically, let $\pi$ be a scheduling policy which chooses a job-processor matching $\pi_t(s_t)\in\Ascr(s_t)$ at $t$, and let $\Pi$ be the set of all such policies.  Define the expected total reward-to-go under a policy $\pi$ starting at state $s \in\Sscr$ at time slot $t$, as
\begin{equation}
\begin{split}
J^\pi_t(s) = E\left[\sum_{t'=t}^T R_{t'}(s_{t'},\pi_{t'}(s_{t'})) | s_t = s \right]
\end{split}
\end{equation}
where $T$ is the (random) time where all jobs have completed execution. $T$ may depend on the policy $\pi$ used.  Note that if we wanted to consider a finite, deterministic horizon $\tilde{T}$, we could appropriately generate a schedule based on the modified, truncated reward functions, $\tilde{w}_j(t)$, such that for all $t \leq \tilde{T}$, $\tilde{w}_j(t) = w_j(t)$, otherwise $\tilde{w}_j(t) = 0$.  The expectation is taken over the random service times $\sigma_j$ of the jobs. We let
\begin{equation}
J_t^*(s) = \max_{\pi \in \Pi} J_t^{\pi}(s)
\end{equation}
denote the expected total reward-to-go under the optimal policy, $\pi^* = \argmax_{\pi \in \Pi} J_t^{\pi}(s)$.

The optimal reward-to-go function (or \emph{value} function) $J^*$ and the optimal scheduling policy $\pi^*$ can in principle be computed via dynamic programming.  Once all jobs have been completed, ${\bf x} = \mathbf{0}$ and   no more reward can be earned.  Therefore, $J_t^*(s) = 0$ for all $s=({\bf x}, {\bf y},{\bf z})$ such that ${\bf x} = \mathbf{0}$.

Given the current state $s_t=s$ and the matching $A$ between free jobs and processors enabled at the beginning of time slot $t$, the system will transition to state $s_{t+1}=s'$ at the beginning of time slot $t+1$ with probabilities $P_A(s_{t+1}=s'|s_t=s)$. For example, if the service times $\sigma_j$ are geometrically distributed with probabilities $p_j$ correspondingly, and the system is in state $s_t=s=({\bf x}, {\bf y},{\bf z})$ and matching $A\in\Ascr(s_t)$ is chosen, then
the system transitions to state $s_t=s'=({\bf x}',{\bf y}',{\bf z}')$ with the following probabilities:
\begin{eqnarray}
P_A(x'_j = 0|s) &=& \left\{
                    \begin{array}{ll}
                      1, & \hbox{if $x(s)_j = 0$;} \\
                      p_j, & \hbox{if $y_j(s) < t$ or $(j,n) \in A$ for some $n$;} \\
                      0, & \hbox{otherwise.}
                    \end{array}
                  \right. \nonumber \\
P_A(x'_j = 1|s) &=& \left\{
                    \begin{array}{ll}
                      1, & \hbox{if $x(s)_j = 1$ and $(j,n) \not \in A$ for all $n$;} \\
                      1-p_j, & \hbox{if $y_j(s) < t$ or $(j,n) \in A$ for some $n$;} \\
                      0, & \hbox{otherwise.}
                    \end{array}
                  \right.\nonumber \\
P_A(y'_j = t|s) &=& \left\{
                    \begin{array}{ll}
                      1, & \hbox{if $(j,n) \in A$ for some $n$;} \\
                      0, & \hbox{otherwise.}
                    \end{array}
                  \right. \nonumber \\
P_A(y'_j = y_j|s) &=& \left\{
                    \begin{array}{ll}
                      1, & \hbox{$(j,n) \not \in A$ for all $n$;} \\
                      0, & \hbox{otherwise.}
                    \end{array}
                  \right.\nonumber \\
P_A(z'_n = j|s) &=& \left\{
                    \begin{array}{ll}
                      1-p_j, & \hbox{if $z(s)_n = j$ or $(j,n) \in A$ for some $n$;} \\
                      0, & \hbox{otherwise.}
                    \end{array}
                  \right. \nonumber \\
P_A(z'_n = 0|s) &=& \left\{
                    \begin{array}{ll}
                      p_j, & \hbox{if $z(s)_n = j$ or $(j,n) \in A$ for some $n$;} \\
                      1, & \hbox{if $z(s)_n = 0$ and $(j,n) \not \in A$ for all $j$;} \\
                      0, & \hbox{otherwise.}
                    \end{array}
                  \right.
\end{eqnarray}

We can now recursively obtain $J^*$  using the Bellman recursion
\begin{eqnarray}\label{eqn:DP}
J_t^*(s) &= &\max_{A \in \Ascr} \Big \{  E \big[\sum_{(j,n)\in A} w_j(t+\sigma_j) + \sum_{s'\in\cal S} P_A[s_{t+1}=s'|s_t=s] J_{t+1}^*(s') \Big \} \nonumber \\
&= &\max_{A \in \Ascr}\Big \{E[R_t(s,A)] + J_{t+1}^*(\tilde{S}(s,A)) \big]\Big \}
\end{eqnarray}
where $\tilde{S}(s,A)$ is the random next state encountered given that we start in state $s$ and action $A$ is taken.

The solution can be found using the \textit{value iteration} method.

\begin{proposition} \label{prop:termination}
There exists an optimal control solution to
(\ref{eqn:DP}) which is obtainable via value iteration.
\end{proposition}

\begin{proof}Once the queue is emptied, Bellman's recursion
terminates.  When $\mathbf{x} = \mathbf{0}$, there are no more
jobs left to be processed. No action can generate any reward and the
optimal policy will never leave this state once it reaches it. There exists a
policy which will complete all jobs and cause the
Bellman's recursion to terminate in finite time.  (i.e. we process \emph{all} jobs on a single server, $n$, in random order.  Because $P(\sigma_j < \infty) = 1$, all jobs will be completed in finite time.) This
guarantees the existence of a stationary optimal policy which is
obtainable via value iteration \cite{bertsekas}.
\end{proof}

Of course, this approach is computationally intractable: the state space (the set of all $(\mathbf{x},\mathbf{y},\mathbf{z})$) is exponentially large. As such, this makes such problems pragmatically difficult.

\subsection{A Hardness Result}
We now show that a special case of the non-preemptive scheduling problem is NP-hard.  Consider a deterministic version of the problem where the completion time of job $j$ is $\sigma_j$ with probability $1$.  Let $w_j(t) = v_j$ for $t \leq K$ and $w_j(t) = 0$ otherwise.  We can think of $v_j$ as the value of job $j$ and $K$ as the shared deadline amongst all jobs.
This version of the non-preemptive scheduling problem with decaying rewards can be reduced to the 0/1 Multiple-Knapsack Problem which is known to be NP-complete.
\begin{theorem}\label{th:NPhard}
The non-preemptive scheduling problem with decaying rewards is NP-hard.
\end{theorem}
\begin{proof}
In the case of the 0/1 Multiple Knapsack Problem, there are $J$ objects of size $\sigma_j$ and value $v_j$ to be placed in $N$ knapsacks of capacity $K$.  Reward is only accrued if the entire object is placed in a knapsack--fractional objects are not possible.  The optimal packing of objects is equal to the optimal scheduling policy for the non-preemptive scheduling with decaying rewards problem.  This reduction takes constant time.
This completes the proof.
\end{proof}

\section{A Greedy Heuristic for Non-preemptive Scheduling of Decaying Jobs}\label{sec:greedy}
In light of Theorem \ref{th:NPhard}, finding an optimal policy for the scheduling problem at hand is computationally intractable.  Therefore, it is highly desirable to find simple, but effective heuristics for practical deployment.  In this section, we examine one such policy.  

A natural heuristic policy one may consider for the stochastic depletion problem is given by the greedy policy which in state $s$ with $F = \sum_n \1_{\{z(s)_n = 0\}}$ free processors chooses the $F$ available jobs with maximum expected utility rate earned over the following time-step, $\frac{E[w_j(t(s)+\sigma_j)]}{E[\sigma_j]}$. That is
\begin{equation}
\begin{split}
\pi_t^{g}(s) = \argmax_{A \in \Ascr} \sum_{(j,n)\in A} \frac{E[w_j(t(s)+\sigma_j)]}{E[\sigma_j]}
\end{split}
\end{equation}
Such a policy is \emph{adaptive} but ignores the evolution of the reward functions, $w_j(t)$, and its impact on rewards accrued in future states.  We denote by $J_t^g(s)$ the reward garnered by the greedy policy starting in state $s$.

\subsection{Sub-optimality of Greedy Policy}
We start with  an instructive example which demonstrates the nature (and degree) of sub-optimality of the greedy policy.

\begin{example}(Greedy Sub-Optimality)
\label{ex:suboptimality}
Consider the case with $2$ jobs and $1$ machine.  Time is initialized to $0$ so that $t=0$, $J = 2$ and $N = 1$. Assume that each job is waiting to begin processing so that $x_1 = x_2 = 1$ and $y_1 = y_2 = \emptyset$. The service times are Geometric and the expected service times for job $1$ and $2$ are $M$ and $1$, respectively, i.e. $p_1 = \frac{1}{M}$ and $p_2 = 1$.  The reward functions are:
\[
\begin{split}
{\rm For} \ j=1:& \
w_j(t) =
\begin{cases}
M^2, & \text{$t = 1$}\\
0, & \text{$t > 1$}
\end{cases}
\\
{\rm For} \ j=2:& \
w_j(t) =
\begin{cases}
1+\epsilon, & \text{$\forall t$}
\end{cases}
\end{split}
\]
for $\epsilon > 0$. Hence, the completion of job $1$ generates rewards of $M^2$ if it is completed in the first time slot; otherwise, no revenue is received. On the other hand, job $2$ results generates reward of $1+\epsilon$, regardless of which time slot it is completed in.  Therefore, the reward rates are:
\begin{eqnarray}
\frac{E[w_1(t+\sigma_1)]}{E[\sigma_1]} & =
\begin{cases}
1, & \text{$t = 1$}\\
0, & \text{$t > 1$}
\end{cases}
\nonumber\\
\frac{E[w_2(t+\sigma_2)]}{E[\sigma_2]} & =
\begin{cases}
1+\epsilon, & \text{$\forall t$}
\end{cases} \nonumber
\end{eqnarray}
In time slot $t = 0$, the greedy policy schedules job $2$ because its reward rate ($1+\epsilon$) is great than that of job $1$ ($1$). Job $2$ completes processing in one time slot and generates reward of $w_2(1) = 1 + \epsilon$.  At $t=1$, only job $1$ remains to be processed.  However, the service time for job $1$ is at least one time slot, so when job $1$ is completed at $t = 1+\sigma_2>1$, $0$ reward is generated.  Hence the greedy policy generates a total expected reward of  $1+\epsilon$.

On the other hand, the optimal policy realizes the reward of job $1$ is degrading and schedules it first.  With probability $1/M$, job $1$ will complete by time slot $t = 1$ and generate reward $M^2$.  However, with probability $1-1/M$ it will take more than one time slot and generate no reward since $w_1(t) = 0$ for $t > 1$. Upon the completion of job $1$,  job $2$ is scheduled and it completes processing in $1$ time slot.  Since $w_2(t) = 1 + \epsilon$ for all $t$, this results in additional reward of $1+\epsilon$.   Hence, the total expected reward generated by the optimal policy is $M+1+\epsilon$.  Comparing the performance of the optimal and greedy policies gives $J^*(s)/J^g(s) = (M+1+\epsilon)/(1 + \epsilon)$.
\end{example}
Letting $\epsilon \rightarrow 0$, it is easy to see that the greedy policy results in an $M+1$ approximation, where $M = \frac{E[\max_j \sigma_j]}{\min_j E[\sigma_j]} = \Delta$.  This suggests that the approximation of the greedy policy is dependent on the relationship between job service times.  The following subsection specifies this relation.

%
\subsection{The Greedy Heuristic is an online $(2+\Delta)$-Approximation}\label{ssec:bound}
In this section we will show that the greedy heuristic is within a factor of $2+\Delta$ of optimal, where
\begin{equation}\label{eq:delta}
\Delta =\frac{E[\sigma_{\max}]}{\min_jE[\sigma_j]}
\end{equation}
Before we can prove this result, we need to first show a few properties of the system and the optimal value function, $J_t^*$.

We begin with a monotonicity property based on the number of jobs remaining to be processed.  Intuitively, if one were given an additional set of jobs to process, the reward that can be garnered by the completion of the original set of jobs in conjunction with the additional jobs will be more than if those extra jobs were not available.  Consider two states: $s$ and $s'$ which are nearly identical except state $s$ has more jobs to process than state $s'$. In other words, all jobs that have been completed in the $s$-system have also been completed in the $s'$-system.   Similarly, any job that has started processing in the $s$-system has also started processing in the $s'$-system at the exact same time on the same machine.  Any additional jobs in state $s$ are jobs that have not started processing, but have already been completed in state $s'$. That is, the additional jobs are only available for processing in the $s$-system. Then the reward-to-go generated starting in state $s$ is larger than that starting in state $s'$.   The following lemma formalizes this intuition.

\begin{lemma}(Monotonicity in Jobs)\label{lemma:mono_job}
Consider states $s$ and $s'$ such that state $s$ has more jobs than state $s'$  and any job that has started in state $s'$ started processing in the exact same time slot in state $s$ so that for each job $j$: $x(s)_j \geq x_j(s')$ and
\begin{eqnarray}
y(s)_j = \left\{
           \begin{array}{ll}
             y(s')_j, & \hbox{if $x(s)_j = x(s')_j$;} \\
             \emptyset, & \hbox{if $x(s)_j > x(s')_j$.}
           \end{array}
         \right. \nonumber
\end{eqnarray} Also, in both states, each processor $n$ is either not busy or busy processing the same job: $z(s)_n = z(s')_n$.  For all states $s$ and $s'$ which satisfy these conditions, the following holds:
$$J_t^*(s) \geq J_t^*(s').$$
\end{lemma}

\begin{proof}
Consider a coupling of the systems starting at $s$ and $s'$ such that they see the same realizations of service times $\sigma_j$ (and residual service times for jobs that have already started processing).  This is possible for all jobs $j \in \Jscr_{s'} = \{j \in \Jscr|x(s')_j = 1\} \subseteq \Jscr_s = \{j\in \Jscr|x(s)_j = 1\}$ because they have the same distributions.   $\Jscr_s$ and $\Jscr_{s'}$ denote the jobs to be completed under the systems starting in states $s$ and $s'$, respectively.

Let $\pi^*(s')$ denote the optimal scheduling policy starting from state $s'$.
Consider a policy $\tilde{\pi}$ that starts in state $s$ and mimics $\pi^*(s')$ until all jobs $j \in \Jscr_{s'}$ are completed and completes the rest of the jobs $j \in \Jscr_s \setminus \Jscr_{s'}$ in sequential order.  That is, under  $\tilde{\pi}$ the scheduler initially pretends that jobs $j\in  \Jscr_s \setminus \Jscr_{s'}$ do not exist and uses the optimal policy under this assumption; once these jobs are completed, it processes the remaining jobs in an arbitrary order.  Said another way, the $\tilde{\pi}$ policy blocks processing of the additional jobs in state $s$ ($j\in  \Jscr_s \setminus \Jscr_{s'}$) and optimally processes the remaining jobs.  Once these jobs ($j\in \Jscr_{s'}$) are completed, the $\tilde{\pi}$ policy `unlocks' the remaining, additional jobs and processes them in an arbitrary manner.  Fig. \ref{fig:monojobs} demonstrates the relationship between $\pi^*(s')$ and $\tilde{\pi}$ for a single server over a particular sample path for service times.

\begin{figure}[htb]
\begin{center}
\psfrag{spsys}{$s'$-system}
\psfrag{ssys}{$s$-system}
\psfrag{s1}{$\sigma_1$}
\psfrag{s2}{$\sigma_2$}
\psfrag{s3}{$\sigma_3$}
\psfrag{s4}{$\sigma_4$}
\psfrag{s5}{$\sigma_5$}
\psfrag{s6}{$\sigma_6$}
\psfrag{t}{$t$}\psfrag{t1}{$T_{s'}$}\psfrag{t2}{$T_{s}$}
\includegraphics[scale=.8]{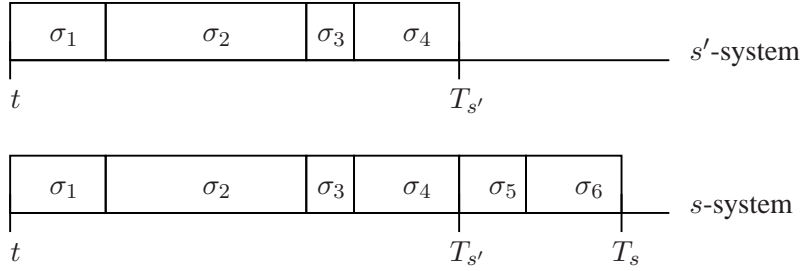}
\caption{Monotonicity in Jobs: A single server scenario.  The $s$-system is given $2$ additional jobs, $j =5$ and $j = 6$.  The $s'$-system uses policy $\pi^*(s')$ to optimally process all jobs $j = 1,2,3,4$. The $s$-system uses policy $\tilde{\pi}$ which mimics $\pi^*(s')$ until all jobs $j = 1,2,3,4$ are completed at time $T_{s'}$ and then processes the remaining additional jobs. }\label{fig:monojobs}
\end{center}
\end{figure}

Let $T_j$ be the completion time of job $j$ for the $s$-system when using policy $\tilde{\pi}$. Similarly, let $T_j^*$ be the completion time of job $j$ for the $s'$-system under the optimal policy, $\pi^*(s')$.  By our coupling, for all $j \in \Jscr_{s'}$, $T_j = T^*_j$, i.e. the completion time of job $j$ is identical under the $s$-system which uses policy $\tilde{\pi}$ and under the $s'$-system which uses policy $\pi^*(s')$. (Notice in Fig. \ref{fig:monojobs}, jobs $1,2,3,4$ complete at the same time in the $s'$ and $s$-systems).  We use the notation $J_t^*(s|\sigma)$ as the optimal reward-to-go given the filtration of the job service times, i.e. given a sample path of realizations of the $\sigma_j$.  We employ similar notation for $J^{\tilde{\pi}}_t$. We have,
\begin{eqnarray}
J_t^*(s|\mathbf{\sigma}) & \geq & J_t^{\tilde{\pi}}(s|\mathbf{\sigma}) \nonumber\\
& = & \sum_{j \in \Jscr}w_j(T_j) \nonumber \\
& = & \sum_{j \in \Jscr_{s'}}w_j(T_j) + \sum_{j \in \Jscr_s \setminus \Jscr_{s'}}w_j(T_j) \nonumber \\
& = &  J_t^*(s'|\mathbf{\sigma})  + \sum_{j \in \Jscr_s \setminus \Jscr_{s'}}w_j(T_j) \nonumber \\
& \geq & J_t^*(s'|\mathbf{\sigma})\nonumber
\end{eqnarray}
The first inequality comes from the optimality of $J_t^*(\cdot)$. The first equality comes from the definition of the reward function, $T_j$, and $\tilde{\pi}$ policy.  The third equality comes from the coupling of the two systems so that $T_j = T_j^*$ for all $j\in \Jscr_{s'}$. The last inequality comes from non-negative property of the rewards in Assumption \ref{as:reward}.  Taking expectations over $\sigma_j$ yields the desired result.
\end{proof}

Next, we consider a property of the optimal policy.  In every time slot, there will be a set (possibly empty) of free machines ($z_n = 0$).  In each time slot,  the optimal policy  will assign a job to all \emph{free} machines, assuming there are enough available jobs. That is, while there are still jobs waiting to be processed, no machine will idle under the optimal policy.
\begin{lemma}(Non-idling)\label{lemma:freemach} Suppose in state $s$, there are $F = |\{n \in \Nscr|z(s)_n = 0\}|$ free machines, and the number of jobs remaining to be processed is $K = |\{j \in \Jscr| x(s)_j = 1, y(s)_j = \emptyset\}|$.  Then, under the optimal policy $\pi^*(s)$, the number of job-processor pairs executed in the next time slot will be:
$$|A| = \min\{K,F\}.$$
i.e. the optimal policy is non-idling.
\end{lemma}
\begin{proof}
The proof is by contradiction.  What needs to be shown is that nothing can be gained by idling ($|A| < \min\{K,F\}$).  Suppose that under the optimal policy, a processor remains free (idles), even though there is an available job to work on. Consider another policy $\tilde{\pi}$ which is identical to the $\pi^*$ policy except it begins processing all jobs on the idling machine one time slot earlier.  Due to assumption \ref{as:reward}, by processing the jobs earlier, this will result in an increase in reward.  This contradicts the optimality of the idling policy; hence, no optimal policy will idle.
\end{proof}

Now consider two systems which are identical, except one machine is tied up longer in the second system.  The following lemma says that the maximum amount of additional revenue accrued by the first system for being able to start processing earlier is given by the reward rate of the greedy job; that is, the job of maximum reward rate amongst those in processing or waiting.
\begin{lemma}(Greedy Revenue)\label{lemma:scaled_revenue}
Consider a state $s_t = s$ in time slot $t$ and let $g$ denote the index of a greedy job, i.e. $g = \argmax_{j\in \Jscr_s} \frac{E[w_j(t+\sigma_j)]}{E[\sigma_j]}$ for all jobs which are mid-processing or have not started ($\Jscr_s = \{k \in \Jscr|x_k(s) = 1\}$).

Denote by $s_g$ and $s_i$ two states which are related to state $s$ in the following manner.  The two states are identical to state $s$, except on free machine $n_g$ ($z(s)_{n_g} = 0$).  In state $s_g$, machine $n_g$ is occupied by a {\bf replica} of job $g$ meaning it has the same service time as job $g$, however, its completion does not generate any rewards nor does it effect the completion of the original job $g$.  Similarly, in state $s_i$, machine $n_g$ is occupied by a replica of job $i$.  Said in notation:  $x(s_i)_j = x(s_g)_j = x(s)_j$ and $y(s_i)_j = y(s_g)_j = y(s)_j$ for all $j$, $z(s_i)_n = z(s_g)_n = z(s)_n$ for all $n \not = n_g$, and $z (s_g)_{n_g} = g$ while $z(s_i)_{n_g} = i$ for some arbitrary job index $i$ and machine $n_g$.  Then,
$$E[J_t^*(s_i)] \leq \Big(1-\frac{E[\sigma_i]}{E[\sigma_g]} + \frac{E[\max_{j\in\Jscr_s}\sigma_j]}{E[\sigma_g]}\Big)E[w_g(t+\sigma_g)] + E[J_t^*(s_g)]$$
\end{lemma}

\begin{proof}
We begin by coupling the systems such that they see the same realizations for service times. Note that the replicated jobs which currently occupy machine $n_g$ need not have the same service time of their original jobs, $i$ or $g$--despite having the same distribution.

Consider a policy $\tilde{\pi}$ for the $s_g$-system which mimics the $\pi^*(s_i)$ policy.  While processor $n_g$ is occupied by replica job $g$, which blocks processing of other jobs, the $s_g$-system will simulate the service time of jobs on processor $n_g$.  There are two possible cases, $\sigma_i \geq \sigma_g$ and $\sigma_i < \sigma_g$.

\begin{description}
  \item[Case 1, $\sigma_i \geq \sigma_g$:] the $\tilde{\pi}$ policy idles on machine $n_g$ until $t+\sigma_i$ (time which machine $n_g$ is free in the $s_i$-system).  At this point, the $s_g$-system is `synced' with the $s_i$-system and it proceeds with executing the optimal policy for the $s_i$ system, $\pi^*(s_i)$.  See Fig. \ref{fig:case1} for a single processor example of such a scenario.

\begin{figure}[htb]
\begin{center}
\psfrag{sisys}{$s_i$-system}
\psfrag{sgsys}{$s_g$-system}
\psfrag{si}{$\sigma_i$}
\psfrag{sg}{$\sigma_g$}
\psfrag{s2}{$\sigma_2$}
\psfrag{s3}{$\sigma_3$}
\psfrag{s4}{$\sigma_4$}
\psfrag{s5}{$\sigma_5$}
\psfrag{s6}{$\sigma_6$}
\psfrag{t}{$t$}\psfrag{tg}{$t+\sigma_g$}
\psfrag{t1}{$t+\sigma_i$}
\psfrag{t2}{$T_{s_i}$}
\includegraphics[scale=.8]{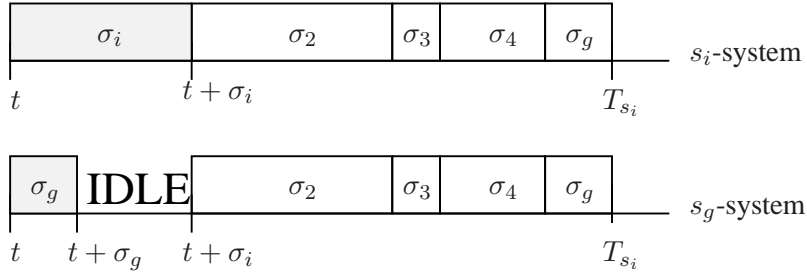}
\caption{Case 1, $\sigma_i \geq \sigma_g$.  The optimal policy is used for the $s_i$-system, which processes jobs in order $2,3,4,g$.  The $s_g$-system uses policy $\tilde{\pi}$ which mimics $\pi^*(s_i)$. Because job $i$ completes after job $g$, the $\tilde{\pi}$ policy idles.  Note that job $g$ is processed twice in the $s_g$ system because the first job is just a replica.  Job $i$ is only processed once in the $s_i$ system because even though $i$ is a replica, the original had already completed processing.  }\label{fig:case1}
\end{center}
\end{figure}

If $T_j^*(s_i)$ is the completion time of job $j$ in the $s_i$-system under  optimal policy $\pi^*(s_i)$,and $T_j$ is the completion time of job $j$ in the $s_g$-system under the $\tilde{\pi}$ policy, then $T_j = T_j^*(s_i)$.  Employing similar notation as before, we consider the reward-to-go on a single realized sample path of service times, given by $\sigma$ and the event $\sigma_i \geq \sigma_g$:
\begin{eqnarray}\label{eqn:gr1}
J_t^*(s_i|\sigma, \sigma_i \geq \sigma_g) & = & \sum_{j\in \Jscr_s} w_j(T_j^*(s_i))\\
& = &\sum_{j\in \Jscr_s} w_j(T_j)\nonumber \\
& = & J_t^{\tilde{\pi}}(s_g|\sigma,\sigma_i \geq \sigma_g) \nonumber \\
& \leq & J_t^*(s_g|\sigma,\sigma_i \geq \sigma_g) \nonumber \\
& \leq & J_t^*(s_g|\sigma,\sigma_i \geq \sigma_g) +  \frac{E[w_g(t+\sigma_g)|\sigma,\sigma_i < \sigma_g]}{E[\sigma_g|\sigma,\sigma_i < \sigma_g]} E[\sigma_g-\sigma_i+\sigma_{\max}|\sigma,\sigma_i < \sigma_g]  \nonumber
\end{eqnarray}

  \item[Case 2, $\sigma_i < \sigma_g$:] In this case, $\tilde{\pi}$ cannot exactly mimic $\pi^*(s_i)$ policy because  machine $n_g$ will continue to be busy after $i$ completes in the $s_i$-system. The $\tilde{\pi}$ policy will simulate the processing of jobs on $n_g$, while the machine is still busy. Let $\Jscr_{sim}$ denote the set of jobs whose processing is simulated.  Despite the fact that these simulated jobs will not actually be completed, the $\tilde{\pi}$ policy assumes they are.  The $\tilde{\pi}$ policy continues to follow the $\pi^*(s_i)$ policy until all jobs are `completed' in the sense that they are actually completed or their completion was simulated because processor $n_g$ was busy under the $s_g$-system when it was free under the $s_i$-system.  The $\tilde{\pi}$ policy then finishes processing the simulated jobs ($j \in \Jscr_{sim}$) in an arbitrary manner so that they are actually completed. That is, the actual completion of the simulated jobs is transferred to after the rest of the jobs have completed processing.  Fig. \ref{fig:case2} shows an example sample path of this scenario.

\begin{figure}[htb]
\begin{center}
\psfrag{sisys}{$s_i$-system}
\psfrag{sgsys}{$s_g$-system}
\psfrag{si}{$\sigma_i$}
\psfrag{sg}{$\sigma_g$}
\psfrag{s2}{$\sigma_3$}
\psfrag{s3}{$\sigma_4$}
\psfrag{s4}{$\sigma_5$}
\psfrag{s5}{$\sigma_2$}
\psfrag{t}{$t$}\psfrag{tg}{$t+\sigma_g$}
\psfrag{t1}{$t+\sigma_i$}
\psfrag{t3}{$\tau$}
\psfrag{sim}{simulated jobs}
\psfrag{t2}{$T_{s_i}$}
\psfrag{t4}{$T_{s_g}$}
\includegraphics[scale=.8]{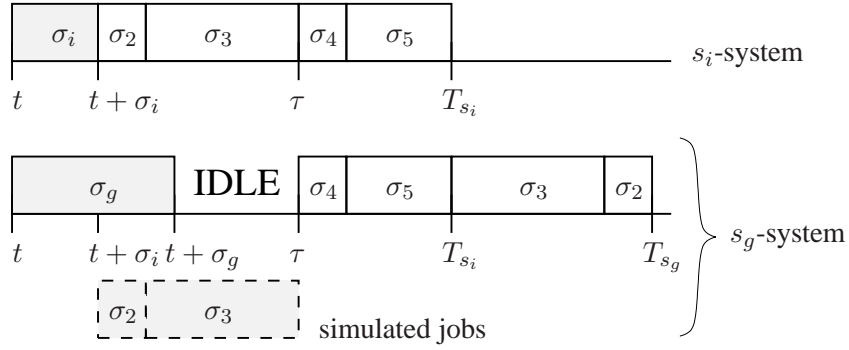}
\caption{Case 2, $\sigma_i < \sigma_g$.  The optimal policy is used for the $s_i$-system, which processes jobs in order $2,3,4$.  The $s_g$-system uses policy $\tilde{\pi}$ which mimics $\pi^*(s_i)$. Because job $i$ completes before job $g$, the $\tilde{\pi}$ policy is blocked until $t + \sigma_g$.  At time $t+\sigma_i$, the $\tilde{\pi}$ policy simulates the processing of jobs $2$ and $3$ on machine $n_g$.  The machine will idle once replica  job $g$ completes and before job $3$ finishes its simulated processing.  At time $\tau$, the $\tilde{\pi}$ policy is able to follow the $\pi^*(s_i)$ policy.  Then the simulated jobs $2$ and $3$ are completed in an arbitrary order after the $\pi^*(s_i)$ policy completes at time $T_{s_i}$. Note that job $g$ and $i$  are processed once in each system because the original jobs have already completed processing (the replicas are processed by time $t$).  }\label{fig:case2}
\end{center}
\end{figure}

      If $T_j^*(s_i)$ is the completion time of job $j$ in the $s_i$-system under
      optimal policy, $\pi^*(s_i)$, and $T_j$ is the completion time of job $j$ in the    $s_g$-system under the $\tilde{\pi}$ policy, then $T_j = T_j^*(s_i)$ for all $j\not \in \Jscr_{sim}$.  Then (again employing the notation given the filtration of $\sigma_j$ and the case $\sigma_i < \sigma_g$):
\begin{eqnarray}\label{eqn:gr_help}
J_t^*(s_i|\sigma,\sigma_i < \sigma_g)& = &  \sum_{j \in \Jscr_s} w_j(T_j^*(s_i))  \nonumber \\
 & = & \sum_{j \in \Jscr_{sim}} w_j(T_j^*(s_i)) + \sum_{j \not \in \Jscr_{sim}} w_j(T_j^*(s_i))  \nonumber \\
 & \leq & \sum_{j \in \Jscr_{sim}} w_j(T_j^*(s_i)) + \sum_{j \not \in \Jscr_{sim}} w_j(T_j^*(s_i)) + \sum_{j \in \Jscr_{sim}} w_j(T_j')  \nonumber \\
& = &\sum_{j \in \Jscr_{sim}} w_j(T_j^*(s_i)) + J_t^{\tilde{\pi}}(s_g|\sigma,\sigma_i < \sigma_g)\nonumber \\
& \leq &\sum_{j \in \Jscr_{sim}} w_j(T_j^*(s_i)) + J_t^*(s_g|\sigma,\sigma_i < \sigma_g)
\end{eqnarray}
The first inequality comes from the non-negativity of rewards.  The third equality comes from our coupling and the definition of the $\tilde{\pi}$ policy.  The last inequality comes from the optimality of $J^*_t$.

Taking expectations over the $\sigma_j$, or equivalently the $T^*_j(s_i)$, and using a little algebra for (\ref{eqn:gr_help}):
\begin{eqnarray}\label{eqn:gr2}
J_t^*(s_i|\sigma_i < \sigma_g)
& \leq &\sum_{j \in \Jscr_{sim}} E[w_j(T_j^*(s_i))|\sigma_i < \sigma_g] + J_t^*(s_g|\sigma_i < \sigma_g) \\
& \leq & \sum_{j \in \Jscr_{sim}} E[\sigma_j|\sigma_i < \sigma_g]\frac{E[w_j(t + \sigma_j)|\sigma_i < \sigma_g]}{E[\sigma_j|\sigma_i < \sigma_g]} + J_t^*(s_g|\sigma_i < \sigma_g) \nonumber \\
& \leq &  \max_k \frac{E[w_k(t+\sigma_k)|\sigma_i < \sigma_g]}{E[\sigma_k|\sigma_i < \sigma_g]}\sum_{j \in \Jscr_{sim}} E[\sigma_j|\sigma_i < \sigma_g] + J_t^*(s_g|\sigma_i < \sigma_g)\nonumber \\
& \leq &\frac{E[w_g(t+\sigma_g)|\sigma_i < \sigma_g]}{E[\sigma_g|\sigma_i < \sigma_g]} E[\sigma_g  + \sigma_l- \sigma_i|\sigma_i < \sigma_g]+  J_t^*(s_g|\sigma_i < \sigma_g)\nonumber \\
& \leq & \frac{E[w_g(t+\sigma_g)|\sigma_i < \sigma_g]}{E[\sigma_g|\sigma_i < \sigma_g]} E[\sigma_g-\sigma_i+\sigma_{\max}|\sigma_i < \sigma_g] + J_t^*(s_g|\sigma_i < \sigma_g)\nonumber
\end{eqnarray}
The second inequality comes from the fact that for all $j$, $T^*_j(s_i) \geq t + \sigma_j$ since the earliest time a job can begin processing is $t$ and all $w_j(t)$ are non-increasing in $t$ (Assumption \ref{as:reward}).  The forth inequality comes from the definition of job $g$. Now, consider the total service time of simulated jobs.  Simulated jobs begin at $t+\sigma_i$ and finish at $\tau > t+\sigma_g$.  In particular, there exists some $l$ such that the first time machine $n_g$ is free under policy $\tilde{\pi}$ is $\tau < t + \sigma_g + \sigma_l$, i.e. $l$ is the last simulated job (job $2$ in Fig. \ref{fig:case2}).  Hence the total service time of simulated jobs is bounded above by  $(t + \sigma_g + \sigma_l) - (t + \sigma_i)$.  This yields inequality four.
\end{description}

Combining (\ref{eqn:gr1}) and (\ref{eqn:gr2}), and taking expectations over $\sigma_i \geq \sigma_g$ and $\sigma_i < \sigma_g$  yields:
\begin{eqnarray}
J_t^*(s_i)& \leq &\frac{E[w_g(t+\sigma_g)]}{E[\sigma_g]}\Big( E[\sigma_g]-E[\sigma_i]+E[\sigma_{\max}]\Big ) + J_t^*(s_g)\nonumber \\
& = &\Big(1 - \frac{E[\sigma_i]}{E[\sigma_g]} + \frac{E[\sigma_{\max}]}{E[\sigma_i]}\Big)E[w_g(t+\sigma_g)]  + J_t^*(s_g)\nonumber
\end{eqnarray}
which concludes the proof.
\end{proof}

Suppose we were able to process a job without using a machine.  The total reward gained by the use of this `virtual machine' is greater than the reward gained without the use of it.  Define $S': \Sscr \times \Jscr \rightarrow \Sscr$ as the operation/function which reduces state $s$ to state $s_i' = S'(s,i)$ by removing job $i$ which has not yet begun processing in state $s$.  That is, starting in state $s$, select a job $i$ that has not been completed.  Complete job $i$ and generate its associated reward without tying up a processor. Said in notation, $\forall n: z(s_i')_n = z(s)_n$; and $\forall j \not = i$: $x(s_i')_j = x(s)_j$ and $y(s_i')_j = y(s)_j$, but $x (s_i')_i= (x(s)_i - 1)^+$ and $y(s_i')_i \not = \emptyset$ .
\begin{lemma}(Virtual Machine Rewards)\label{lemma:vm_rewards}
For all states $s$ and any job $i$, let state $S'(s,i)$ denote the resulting state if job $i$ were processed without occupying a processor.  Also, reward $w_i$ is generated upon completion. Then:
$$J_t^*(s) \leq E[w_i(t+\sigma_i)] + J_t^*(S'(s,i)),$$
\end{lemma}

\begin{proof}
Consider a coupling of the systems starting in state $s$ and $s'_i = S'(s,i)$ such that they see the same realizations of the service times for all jobs.  Let $\pi^*(s)$ denote the optimal scheduling policy starting from state $s$.

In the $s_i'$-system, we call job $i$ a `fictitious' job.  It is fictitious because it does not actually exist (it has already completed) under the $s_i'$-system.  Consider a policy $\tilde{\pi}$ which assumes that job $i$ is a `real' (available/not processed) job and executes the optimal policy under this assumption, i.e. it at time slot $t$, it assumes it is in state $s$ (rather than $s_i'$) and executes the optimal policy $\pi_t^*(s)$.  When $\tilde{\pi}$ schedules job $i$, there is no job to actually process, so the processor will idle while it simulates the processing time for job $i$ which is identically distributed to $\sigma_i$ under the $s$-system.  See Fig. \ref{fig:vmachine} for a single machine example of the $\tilde{\pi}$ and $\pi^*(s)$ policies given a sample path for service time realizations.

\begin{figure}[htb]
\begin{center}
\psfrag{sisys}{$s$-system}
\psfrag{vm}{virtual machine}
\psfrag{sgsys}{$s'_i$-system}
\psfrag{si}{$\sigma_1$}
\psfrag{s2}{$\sigma_2$}
\psfrag{s3}{$\sigma_3$}
\psfrag{s4}{$\sigma_4$}
\psfrag{s5}{$\sigma_5$}
\psfrag{s6}{$\sigma_6$}
\psfrag{t}{$t$}\psfrag{t2}{$T_{s'}$}
\includegraphics[scale=.8]{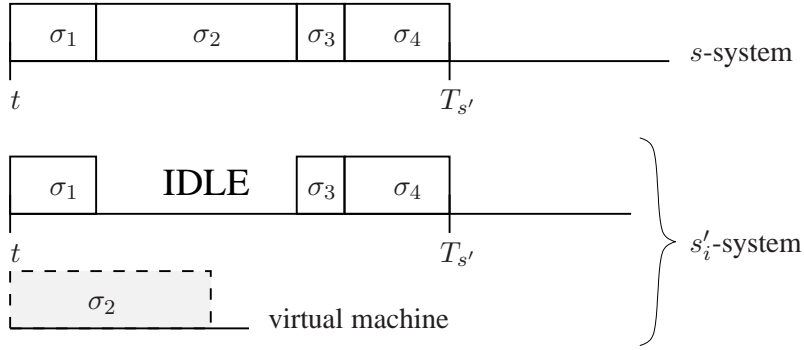}
\caption{Virtual machine: A single server scenario.  Under the $s'_i$-system, job $2$ is processed on a virtual machine.  The $s$-system uses policy $\pi(s')$ to optimally process all jobs $j = 1,2,3,4$. The $s_i'$-system uses policy $\tilde{\pi}$ which mimics $\pi^*(s)$. Because job $j = 2$ has already been processed on the `virtual machine', the $\tilde{\pi}$ policy idles.  }\label{fig:vmachine}
\end{center}
\end{figure}

Let $T_j$ be the completion time of job $j$ under the $\tilde{\pi}$ policy.  Note that $T_i$ is the completion time of the fictitious job, $i$. Let $t_i$ denote the random time which job $i$ begins `processing' under this policy. Under our coupling, $T_j$ is precisely  $t_j$ plus the processing time of job $j$ under $\pi^*(s)$ for the $s$-system. Hence,
\begin{eqnarray}
J_t^*(s|\sigma_j) & = & \sum_{j} w_j(T_j) \nonumber\\
& = & \sum_{j\not = i} w_j(T_j) + w_i(T_i) \nonumber \\
& = &  J_t^{\tilde{\pi}}(s'_i|\sigma_j)  + w_i(t_i + \sigma_i) \nonumber\\
& \leq &  J_t^*(s'_i|\sigma_j)  + w_i(t + \sigma_i) \nonumber
\end{eqnarray}
The inequality results from the non-increasing property of the reward functions in Assumption \ref{as:reward} and from the optimality of $J_t^*(\cdot)$. Taking expectations over $\sigma_j$ yields the desired result.
\end{proof}

We are now in position to prove the main result of this paper. Let $\Delta = \frac{E[\max_j \sigma_j]}{\min_j E[\sigma_j]}$ as in (\ref{eq:delta}).
\begin{theorem}\label{theorem:bound}
For all states $s \in \Sscr$, the following performance guarantee for the greedy policy holds:
$$J_t^*(s) \leq (2+\Delta)J_t^g(s).$$
\end{theorem}
\begin{proof}
The proof proceeds by induction on the number of jobs remaining to be processed, $\sum_{j\in \Jscr} \1_{\{y_j = \emptyset\}}$.  The claim is trivially true if there is only one job remaining to be processed--the greedy and optimal policies will coincide. Now consider a state $s$ such that $\sum_j \1_{\{y(s)_j = \emptyset\}} = K$, and assume that the claim is true for all states $s'$ with $K > \sum_j \1_{\{y(s')_j = \emptyset\}}$.

Now if $\pi_t^*(s) = \pi_t^g(s)$ the then the next state encountered and rewards generated in both systems are
identically distributed so that the induction hypothesis immediately
yields the result for state $s$.

Consider the case where $\pi_t^*(s) \neq \pi_t^g(s)$. Denote by $\Jscr_*$ and $\Jscr_g$ the set of jobs processed by the optimal and greedy policies in state $s$.  Note that these sets depend on the current time slot $t$ and the state $s$; however, we suppress them for notational compactness.   Recall that, by Lemma \ref{lemma:freemach}, $|\Jscr_*| = |\Jscr_g|$. Let $A_*$ and $A_g$ denote the optimal and greedy scheduling policy, respectively, given state $s$ in time slot $t$.

Taking definitions from before, we define $\tilde{S}(s,A)$ as the random next state encountered given that we start in state $s$ and action $A$ is taken.  Also, define $S'(s,i)$ as state $s$ with the completion of job $i$, i.e. job $i$ is completed ($x_i = 0$) without using a processor.

Define the operator $\hat{S} : \Sscr \times \Ascr \rightarrow \Sscr$ which transforms state $s$ by tying up machines with replicated the jobs defined by $A$.  That is, $\hat{s} = \hat{S}(s,A)$ is the state where jobs begin processing on the machines given by $A$, but no reward is generated for their completion and they remain to be processed at a later time (reward is generate upon this second completion).  This second completion may occur prior or following the completion of the replicated job. $A$ defines which jobs are replicated and which machine they are processed on, and hence occupy--replicated jobs do not generated any reward. Put another way, $\hat{s}$ is a new state where machines are occupied for an amount of time defined by the service times of jobs in $A$.  Said in notation, $x(\hat{S}(s,A))_j = x(s)_j$ and $y(\hat{S}(s,A))_j = y(s)_j$ for all $j$,  while $z(\hat{S}(s,A))_n = j$  if $(j,n) \in A$ and $z(\hat{S}(s,A))_n = y(s)_n$ otherwise.

We have:
\begin{eqnarray}
\label{eq:th1}
J_t^*(s) & = & \sum_{j\in\Jscr_*}E[w_j(t+\sigma_j)] + E[J_t^*(\tilde{S}(s,A_*))] \nonumber \\  & \leq & \sum_{(i,g)\in (\Jscr_*,\Jscr_g)}\frac{E[\sigma_i]}{E[\sigma_g]}E[w_g(t+\sigma_g)] + E[J_t^*(\tilde{S}(s,A_*))]\nonumber \\
& \leq & \sum_{(i,g)\in (\Jscr_*,\Jscr_g)}\frac{E[\sigma_i]}{E[\sigma_g]}E[w_g(t+\sigma_g)] +  + E[J_t^*(\hat{S}(s,A_*))]
\end{eqnarray}
The first inequality comes from the definition of the greedy policy; the reward rate for greedy jobs is higher than for the optimal jobs.  The second inequality comes from Lemma \ref{lemma:mono_job} by putting back the jobs in $A_*$.  That is the machines are  occupied by replicas of jobs defined in $A_*$, but the original jobs are placed back to be completed at a later date.  These additional jobs generate more reward as shown in Lemma \ref{lemma:mono_job}.

Continuing (\ref{eq:th1}), we switch $A_*$ with $A_g$.  That is, instead of tying up the machines with replicas of the optimal jobs, they are typed up with replicas of the greedy jobs.  Because $|\Jscr_*| = |\Jscr_g|$ and the processing times on each machines are identical, we can consider each machine individually and use Lemma \ref{lemma:scaled_revenue} so that,
\begin{eqnarray}\label{eq:th2}
\sum_{(i,g)\in (\Jscr_*,\Jscr_g)}\frac{E[\sigma_i]}{E[\sigma_g]}E[w_g(t+\sigma_g)] &+& E[J_t^*(\hat{S}(s,A_*))]\nonumber\\
 & \leq & \sum_{(i,g)\in (\Jscr_*,\Jscr_g)}\frac{E[\sigma_i]}{E[\sigma_g]}E[w_g(t+\sigma_g)] +  E[J_t^*(\hat{S}(s,A_g))] \nonumber\\
&&\sum_{(i,g)\in (\Jscr_*,\Jscr_g)} E[w_g(t+\sigma_g)]\Big(1 - \frac{E[\sigma_i]}{E[\sigma_g]} + \frac{E[\sigma_{\max}]}{E[\sigma_g]}\Big) +  \nonumber \\
& = & \sum_{g\in\Jscr_g} E[w_g(t+\sigma_g)]\Big(1 + \frac{E[\sigma_{\max}]}{E[\sigma_g]}\Big) +  E[J_t^*(\hat{S}(s,A_g))]
\end{eqnarray}

Continuing (\ref{eq:th2}), we now complete the greedy jobs without occupying any machines:
\begin{eqnarray}\label{eq:th3}
\sum_{j\in\Jscr_g} E[w_g(t+\sigma_g)]\Big(1 + \frac{E[\sigma_{\max}]}{E[\sigma_g]}\Big) &+  & E[J_t^*(\hat{S}(s,A_g))]\nonumber \\
& \leq & \sum_{g\in\Jscr_g} E[w_g(t+\sigma_g)]\Big(2 +\frac{E[\sigma_{\max}]}{E[\sigma_g]}\Big) +  E[J_t^*(\tilde{S}(s,A_g))]  \nonumber \\
& \leq & \sum_{g\in\Jscr_g}\Big(2 + \frac{E[\sigma_{\max}]}{\min_kE[\sigma_k]}\Big) E[w_g(t+\sigma_g)] + \nonumber \\
 &&\Big(2 +  \frac{E[\sigma_{\max}]}{\min_kE[\sigma_k]}\Big)E[J_t^g(\tilde{S}(s,A_g))]  \nonumber \\
& =& \Big(2+ \Delta\Big)J_t^g(s)\nonumber
\end{eqnarray}
The first inequality comes from use of `virtual machines' for the greedy jobs under Lemma \ref{lemma:vm_rewards}.  The second inequality comes from the induction hypothesis.
This concludes the proof.
\end{proof}

\section{Special Cases}\label{sec:special}
As shown in \cite{gittins,walrand} the greedy policy is optimal, for linear or exponential decaying reward functions.  Under a few other special cases, the bound in Theorem \ref{theorem:bound} can be improved.

\subsection{Identical Processing Times}
Suppose that all job service times are independent and identically distributed, i.e. in the case of Geometric service times, $p_j = p$ for all $j$.  In general, there is no closed form equation for $E[\sigma_{\max}]$; however, in this case, the bound can be improved to a factor of $2$.  To do this, Lemma \ref{lemma:scaled_revenue} needs to be modified.

\begin{lemma}(Greedy Revenue, I.I.D processing times)\label{lemma:scaled_revenue_identical}
Consider a state $s_t = s$ in time slot $t$ and let $g$ denote the index of a greedy job, i.e. $g = \argmax_{j\in \Jscr_s} \frac{E[w_j(t(s)+\sigma_j)]}{E[\sigma_j]}$ for all jobs which are mid-processing or have not started ($\Jscr_s = \{k \in \Jscr|x_k(s) = 1\} $).  Denote by $s_g$ and $s_i$ two states which are related to state $s$ as follows:  $x(s_i)_j = x(s_g)_j = x(s)_j$ and $y(s_i)_j = y(s_g)_j = y(s)_j$ for all $j$, $z(s_i)_n = z(s_g)_n = z(s)_n$ for all $n \not = n_g$, and $z (s_g)_{n_g} = g$ while $z(s_i)_{n_g} = i$ for some arbitrary job index $i$ and machine $n_g$.  That is in state $s_i$, machine $n_g$ is occupied by a replica of job $g$; and in state $s_i$, machine $n_g$ is occupied by a replica of job $i$. Then,
$$E[J_t^*(s_i)] = E[J_t^*(s_g)]$$
\end{lemma}

\begin{proof}
Couple the systems such that they see the same realizations for service times of job $i$ and job $g$ which are currently occupying machine $n_g$.  This coupling is possible since the jobs are i.i.d.   Therefore, under this coupling there is no difference between state $s_i$ and $s_g$ since these `jobs' are only occupying the machine but are not generating any rewards. Hence, $E[J_t^*(s_i)] = E[J_t^*(s_g)]$.
\end{proof}

Now we are able to prove an improved bound on the performance of the greedy policy.
\begin{theorem}\label{theorem:bound_up}
Let the service time for job $j$ be distributed according to density function $f_j(\sigma)$.  If all job service times are independent and identically distributed according to $f(\sigma)$, i.e. $f_j(\sigma) = f(\sigma)$ $\forall j$, then for all states $s \in \Sscr$, the greedy policy is guaranteed to be within a factor of $2$ of optimal:
$$J_t^*(s) \leq 2J_t^g(s).$$
\end{theorem}

\begin{proof}
Under this scenario, Lemma \ref{lemma:scaled_revenue} can be replaced by Lemma \ref{lemma:scaled_revenue_identical} in the proof of Theorem \ref{theorem:bound}.  Hence, $E[J_t^*(\hat{S}(s,A_*))] = E[J_t^*(\hat{S}(s,A_g))]$ since the distribution of completion times is identical, the amount of time a processor is busy is independent of which job it is processing.  Instead of replicating the entire proof here, we examine how (\ref{eq:th1}), (\ref{eq:th2}), and (\ref{eq:th3}) change.

The only difference for (\ref{eq:th1}) is that $E[\sigma_j] = E[\sigma_i]$ for $i,j$ which allows for a slight simplification.
\begin{eqnarray}
J_t^*(s) & = & \sum_{j\in\Jscr_*}E[w_j(t+\sigma_j)] + E[J_t^*(\tilde{S}(s,A_*))] \nonumber \\  & \leq & \sum_{(i,g)\in (\Jscr_*,\Jscr_g)}\frac{E[\sigma_i]}{E[\sigma_g]}E[w_g(t+\sigma_g)] + E[J_t^*(\tilde{S}(s,A_*))]\nonumber \\
& \leq & \sum_{g\in\Jscr_g}E[w_g(t+\sigma_g)] +  + E[J_t^*(\hat{S}(s,A_*))]
\end{eqnarray}
Now, with improvement to Lemma \ref{lemma:scaled_revenue} in Lemma \ref{lemma:scaled_revenue_identical}, (\ref{eq:th2}) is reduced significantly
\begin{eqnarray}
\sum_{g \in \Jscr_g}E[w_g(t+\sigma_g)] + E[J_t^*(\hat{S}(s,A_*))]
  =  \sum_{g\in\Jscr_g}E[w_g(t+\sigma_g)] +  E[J_t^*(\hat{S}(s,A_g))]
\end{eqnarray}
Finally, utilizing Lemma \ref{lemma:vm_rewards} and completing/generating rewards for the greedy jobs gives:
\begin{eqnarray}
\sum_{g\in\Jscr_g}E[w_g(t+\sigma_g)] +  E[J_t^*(\hat{S}(s,A_g))]
& \leq & 2\sum_{g\in\Jscr_g}E[w_g(t+\sigma_g)] +  E[J_t^*(\tilde{S}(s,A_g))]    \\
& \leq & 2\sum_{g\in\Jscr_g}E[w_g(t+\sigma_g)] +  2E[J_g^*(\tilde{S}(s,A_g))]   \nonumber \\
& =& 2J_t^g(s)\nonumber
\end{eqnarray}
\end{proof}

In the case of i.i.d. service times, the greedy policy corresponds to scheduling the job with the highest expected rewards over their identical completion times.  While this seems to be an intuitive policy, the following example shows what can go wrong.

\begin{example}\label{ex:iid}
Consider the case with $2$ jobs and $1$ machine ($J = 2$ and $N = 1$).  We begin at $t=0$. Assume that neither job has begun processing so that $x_1 = x_2 = 1$ and $y_1 = y_2 = \emptyset$. The service times for job $1$ and $2$ are both deterministic and equal to $1$.  The reward functions are:
\[
\begin{split}
{\rm For} \ j=1:& \
w_j(t) =
\begin{cases}
1-\epsilon, & \text{$t = 1$}\\
0, & \text{$t > 1$}
\end{cases}
\\
{\rm For} \ j=2:& \
w_j(t) =
\begin{cases}
1, & \text{$\forall t$}
\end{cases}
\end{split}
\]
for $\epsilon > 0$. So that the completion of job $1$ only results in revenue if it is completed in the first time slot, but job $2$ results in the same revenue, regardless of which time slot it is completed in.  Therefore, the reward rates are:
\begin{eqnarray}
\frac{E[w_1(t+\sigma_1)]}{E[\sigma_1]} & =&
\begin{cases}
1-\epsilon, & \text{$t = 1$}\\
0, & \text{$t > 1$}
\end{cases}
\nonumber\\
\frac{E[w_2(t+\sigma_2)]}{E[\sigma_2]} & =&
\begin{cases}
1, & \text{$\forall t$}
\end{cases} \nonumber
\end{eqnarray}
Clearly, the greedy policy is to schedule job $2$ and then job $1$ since the reward rate for job $2$ is greater than that for job $1$ ($1+\epsilon > 1$). However, when job $1$ completes at $t = 2$, it generates no reward since $w_1(2) = 0$.  This results in  reward $1$.  On the other hand, the optimal policy realizes the reward of job $1$ is degrading and schedules it first and schedules job $2$ second.   This results in reward $2-\epsilon$. We thus see that $J_t^*(s) = (2-\eps) J_t^{g}(s)$ here.
\end{example}
In light of the example just shown, the bound in Theorem \ref{theorem:bound_up} is \emph{tight}.

\subsection{Slowly Decaying Rewards}\label{ssec:asym_opt}
We have proven a worse case bound for arbitrary decaying rewards.  If the time-scale of decay is very  long compared to the time-scale of job completion times, then the rewards would be nearly constant during the processing time of a job.  In particular, as the decay goes to zero over the time-scale of job completion times, the performance of the greedy heuristic approaches the performance of the optimal policy.

We will now formally define the time-scale of decay.  Consider a difference equation specification for the time-scale of decay.  Let
$$\delta = \max_{t,k,m} E[w_k(t) - w_k(t + \sigma_m)] \geq 0.$$
We will show that as $\delta \rightarrow 0$, $J_t^g(s) \rightarrow J_t^*(s)$. To do this, we must start with a few preliminary results.

The first is, as $\delta \rightarrow 0$, rewards become invariant to the completion time.  Rewards are generated upon the completion of each job.  However, as $\delta \rightarrow 0$, the rewards generated at the completion time of a job is nearly the rewards that would have been generated had the job had $0$ processing time.
\begin{lemma}(Time-Invariant Rewards)\label{lemma:tinvariant}
For any jobs $i,j$ and time slot $t$, as the time-scale of decay, $\delta$, approaches $0$, the reward generated for completing job $j$ is invariant to shifts in time by the service time of job $i$, $\sigma_i$. In particular, $$\lim_{\delta \rightarrow 0} E[w_j(t + \sigma_i)] =  w_j(t)$$
\end{lemma}
\begin{proof}
For any job indices $i,j$ and time slot $t$:
\begin{eqnarray}
\big | E[w_j(t + \sigma_i)] -  w_j(t) \big | & \leq &
 \max_{\tau,k,m} \big | E[w_k(\tau + \sigma_m)] -  w_k(\tau) \big | \nonumber \\
& = & \delta
\end{eqnarray}
which implies that $ \big | E[w_j(t + \sigma_i)] -  w_j(t) \big |  \rightarrow 0$ as $\delta \rightarrow 0$.
\end{proof}

Because rewards are nearly constant over the time-scale of job completions, starting a job $\sigma_j$ time slots later does not significantly reduce the aggregate reward accrued.  The following lemma is similar to Lemma \ref{lemma:scaled_revenue} for slowly decaying reward functions.  Define $\hat{S}(s,A)$ as in Section \ref{ssec:bound}, so that $\hat{S}(s,A)$ is the state where jobs are processed on the machines given by $A$, but they are not removed and no reward is generated for this initial processing.  These replica jobs occupy the machines, making them unable to process other jobs in the meantime.  However, they do not generate reward.  In notation, $x(\hat{S}(s,A))_j = x(s)_j$ and $y(\hat{S}(s,A))_j = y(s)_j$ for all $j$, while $z(\hat{S}(s,A))_n = j$ for all $(j,n) \in A$ and $z(\hat{S}(s,A))_n = z(s)_n$ otherwise.
\begin{lemma}(Delayed Machine)\label{lemma:delayedmach}
Let $\hat{s} = \hat{S}(s,A)$ denote the resulting state if machines in $A$ are occupied, but all the jobs have the same (un)processed state as in state $s$.  Then, starting in any state $s$ and given action $A$, the difference in optimal reward-to-go generated in states $s$ and $\hat{s}$  goes to $0$ as the time-scale of decay, $\delta$, goes to $0$, i.e.
$$\big | J_t^*(s) - E[J_t^*(\hat{S}(s,A) ]\big | \rightarrow 0\textrm{ as }\delta \rightarrow 0$$
\end{lemma}
\begin{proof}
To begin, note that  $J_t^*(s) \geq E[J_t^*(\hat{S}(s,A))]$.  To see this, we couple the job completion times.  Let $\tilde{\pi}$ denote a policy starting from state $s$, but mimicking the optimal policy starting from state $\hat{S}(s,A) = \hat{s}$, $\pi^*(\hat{s})$.  Therefore, under the $\tilde{\pi}$ policy, machine $n_k$ will idle for $\sigma_k$ time slots before proceeding if $(k,n_k) \in A$. The $s$-system simply delays processing any new jobs until the replica jobs in the $\hat{s}$-system are completed.  In this case, the completion time for jobs will be identical under the $\tilde{\pi}$ and $\pi^*(\hat{s})$ policies.  Hence $J_t^*(\hat{s}) = J_t^{\tilde{\pi}}(s) \leq J_t^*(s)$, by the optimality of $J_t^*(s)$.

Now to show the convergence result, couple the job completion times under the $s$ and $\hat{s}$-systems.  Let $\sigma_k^* = \max_{(k,n) \in A} \sigma_k$ be the maximal service time for jobs in $A$. Consider a policy $\tilde{\pi}$ for the $\hat{s}$-system which idles for $\sigma^*_k$ time slots and begins processing new jobs at time $t' = t+\sigma_k^*$, but assumes that $t' = t$.  Therefore, $\tilde{\pi}$ coincides precisely with $\pi^*(s)$ shifted in time by $\sigma_k^*$.  In other words, the $\tilde{\pi}$ policy waits until $t'$ at which point all replica jobs are completed and then begins processing new jobs as if no time has passed and $t' = t$.

For the $s$-system, let $T_j^*$ be the completion time of job $j$ under the optimal policy $\pi^*(s)$.  Then $T_j^{\tilde{\pi}} = T^*_j + \sigma_k^*$ is the completion time for job $j$ under the $\tilde{\pi}$ policy. Now, given some $\epsilon > 0$
\begin{eqnarray}
\big | J_t^*(s) - E[J_t^*(\hat{s})] \big | & \leq & \big |J_t^*(s) - E[J_t^{\hat{\pi}}(\hat{s}) ]\big | \nonumber \\
& = &   \Big |E\big [\sum_{j}w_j(T^*_j)\big ] - E\big [\sum_j w_j(T^{\tilde{\pi}}_j)\big ] \Big | \nonumber \\
& = &   \Big |\sum_{j}E\big [w_j(T^*_j) - w_j(T^*_j + \sigma_k^*)\big ]  \Big | \nonumber \\
& = &   \Big |\sum_{j}E\Big[w_j(T^*_j) - E\big[w_j(T^*_j + \sigma_k^*)\big ]\Big]  \Big | \nonumber \\
& \leq & J\delta < \epsilon
\end{eqnarray}
The inequality comes from Lemma \ref{lemma:tinvariant} and because $\delta \rightarrow 0$ there exists $\delta < \epsilon/J$.
\end{proof}

Now, we are in position to prove that the performance of the greedy policy approaches the performance of the optimal policy when the decay of rewards is slow compared to the job completion time.
\begin{theorem}(Slowly Decaying Rewards)\label{theorem:slowdecay}
For any state $s\in \Sscr$, as the time-scale of decay goes to $0$, i.e $\delta \rightarrow 0$, the performance of the greedy policy approaches that of the optimal policy.
$$J_t^g(s) \rightarrow J_t^*(s).$$
\end{theorem}
\begin{proof}
The proof is by induction on the number of jobs remaining to begin processing.  Clearly, when only one job remains the greedy and optimal policies coincide.  Now we assume it is true for $K-1$ jobs remaining and show that it is true for $K$ jobs.

Denote by $\Jscr_*$ and $\Jscr_g$ the set of jobs processed by the optimal and greedy policies in state $s$.  Recall that, by Lemma \ref{lemma:freemach}, $|\Jscr_*| = |\Jscr_g|$. Let $A_*$ and $A_g$ denote the optimal and greedy scheduling policy, respectively. As before, we define $\tilde{S}(s,A)$ which is the next state given we start in state $s$ and take action $A$ and $\hat{S}(s,A)$ which is the state with machines in $A$ occupied by replica jobs which generate $0$ reward.

Suppose we are given $\epsilon > 0$.  Define $\delta_{\epsilon,1}$ such that for all $\delta < \delta_{\epsilon,1}$, $\big | J_t^*(s|\sigma_j) - J_t^*(\hat{S}(s,A)|\sigma_j) \big | < \epsilon/2$; this is possible due to Lemma \ref{lemma:delayedmach}.    Define $\delta_{\epsilon,2}$ such that for all $\delta < \delta_{\epsilon,2}$, $\big | J_t^*(s') - J_t^g(s') \big | < \epsilon/2$ for any $s'$ with $K-1$ jobs remaining; this is possible due to our inductive hypothesis. Then let $\delta_{\epsilon} = \min\{\delta_{\epsilon,1},\delta_{\epsilon,2}\}$.  For any $\delta < \delta_{\epsilon}$:
\begin{eqnarray}
J_t^*(s) &\leq& E[J_t^*(\hat{S}(s,A_g))] +\epsilon/2 \nonumber \\
&\leq&\sum_{j\in\Jscr_g}E[w_j(t+\sigma_j)] + E[J_t^*(\tilde{S}(s,A_g))] +\epsilon/2 \nonumber \\
&\leq&\sum_{j\in\Jscr_g}E[w_j(t+\sigma_j)] + E[J_t^g(\tilde{S}(s,A_g))] +\epsilon \nonumber \\
& = & J_t^g(s) + \epsilon \nonumber
\end{eqnarray}
The first inequality is due to Lemma \ref{lemma:delayedmach}, for state $s$ and action given by $A_g$.  The second inequality is by Lemma \ref{lemma:vm_rewards} for removing the greedy jobs.  The third inequality is by the inductive hypothesis.

By the optimality of $J_t^*$, $\big | J_t^*(s) - J_t^g(s) \big | = J_t^*(s) - J_t^g(s)$.  So for $\delta < \delta_{\epsilon}$,  $\big | J_t^*(s) - J_t^g(s) \big | < \epsilon$, which proves our claim.
\end{proof}

This result is intuitive because as the time-scale of decay becomes negligible to the time-scale of job completion times, rewards can be viewed as essentially constant.  As such, it does not matter which order jobs are completed, since all will be completed.  Hence, any policy, and certainly the greedy policy, is nearly optimal.  However, the convergence rate to optimality will vary across policies.

\section{Performance Evaluation}\label{sec:perfeval}
In the previous sections, we have shown performance guarantees for a greedy policy when scheduling jobs with decaying rewards.  In light of Example \ref{ex:suboptimality} and Theorem \ref{theorem:bound}, the loss in performance due to use of the greedy policy can be at least $\Delta +1$ but  can do no worse that $\Delta + 2$.  In this section, we show that, in practice, the greedy performance is likely to be  much better.

\begin{figure}[hbt]
\begin{center}
\includegraphics[scale=.5]{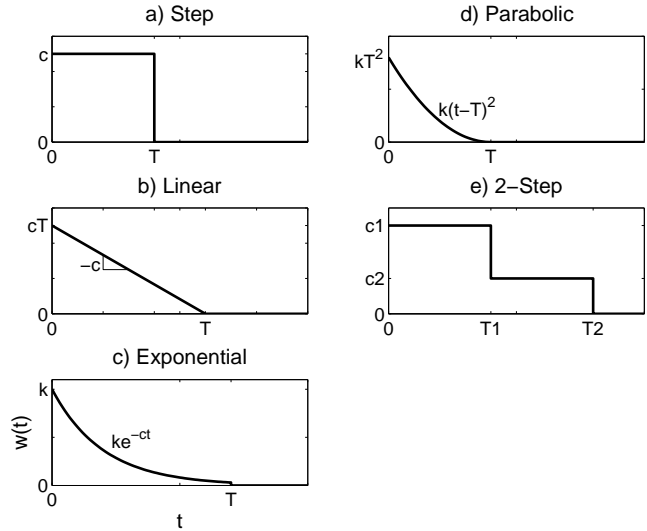}
 \caption{Different types of decaying functions.}\label{fig:decayplots}
\end{center}
\end{figure}

In order to enable computation of an optimal policy we assume that the number of jobs is finite and small
($2$-$10$). Even with a finite number of jobs, $|\Sscr|$ grows exponentially fast in several problem parameters which forces us to limit the size of the problem instances we consider. In particular, we consider problems with a single machine, $M=1$, and jobs with finite deadlines less than $100$.  That is no reward is accrued after $t = 100$. We assume job completion times are Geometric with $p_j$ evenly distributed between $p_{min}$ and $p_{max} = .9$. Since there is no closed form distribution for $\sigma_{\max}$, see Appendix \ref{ap:geo} for how to find an upper-bound to $E[\sigma_{\max}]$ and, subsequently, an upper-bound to $\Delta$. We consider a number of decaying reward functions depicted in Fig. \ref{fig:decayplots}.  The constants defining each reward function are drawn uniformly; all experimental results are averaged over $100$ different realizations of these constants, with $1000$ experiments for each such set.

In Table \ref{table:perf}, we summarize the performance of the greedy policy for the reward functions shown in Fig. \ref{fig:decayplots}.  In this case, $p_{\min} = .1$ and $p_{\max} = .9$; $\min_j E[\sigma_j] = \frac{1}{.9} = 1.11$ and $E[\max_j \sigma_j] < 27.3$, therefore, $2+\Delta< 26.6$.  We can see that while the optimal policy achieves larger reward than the greedy policy, the gains are within a factor of $1.20$--much better than the guarantee provided by Theorem \ref{theorem:bound}.  Because we have finite deadlines for each job, there exists some $T_{max}$ such that for all $j$, $w_j(t) = 0$ for all $t > T_{max}$.  Therefore, the reward function with exponential decay is slightly modified from the standard notion of exponential decay where $w_j(t) \rightarrow 0$, but $w_j(t) > 0$ for any $t < \infty$.  Hence, the greedy policy is \emph{not} optimal for this exponential decay with finite deadline.

\begin{table}[htbp]
\begin{center}
\begin{tabular}{|l|c|c|c|}
  \hline
   &\multicolumn{3}{c|}{$J_t^*/J_t^g$} \\
   \cline{2-4}
  Type & $J = 2$ & $J = 5$ & $J = 8$ \\
  \hline
  \hline
  Step & 1.0065 & 1.0931 & 1.1287 \\
  \hline
  Linear & 1.0133 & 1.0576 & 1.1289 \\
  \hline
  Exponential & 1.0609 & 1.0433 & 1.0590 \\
  \hline
  Parabolic & 1.0265 & 1.0382 & 1.0667 \\
  \hline
  2-step & 1.0218 & 1.1007 & 1.1520 \\
  \hline
\end{tabular}
\end{center}
\caption{Performance of Greedy policy versus Optimal Policy for different types of decaying reward functions for $J$ jobs.} \label{table:perf}
\end{table}
\begin{figure}[htb]
\begin{center}
\psfrag{N}[][][1.1]{$J$}
\psfrag{JoJg}[][][1.4][-90]{
$\frac{J_t^*}{J_t^g}$
}
\includegraphics[scale=.5]{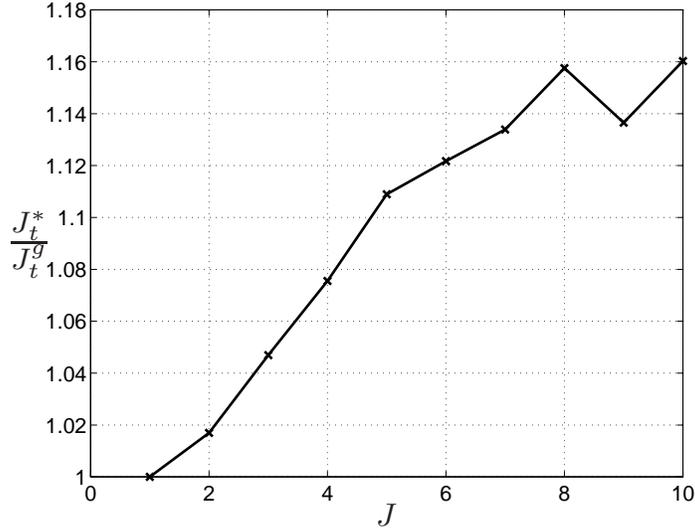}
 \caption{Performance loss ($\frac{J_t^*}{J_t^g}$) as the number of jobs ($J$) increases.}\label{fig:runN}
\end{center}
\end{figure}

It is interesting to note that the performance of the greedy policy seems to degrade as the number of jobs increases.  We examine this more closely in Fig. \ref{fig:runN} under a step function reward function where rewards are constant until a fixed deadline as in Fig. \ref{fig:decayplots}a.  Clearly, the greedy and optimal policies coincide when there is only one job.  As the number of jobs increases, the performance of the greedy policy degrades; however, the loss in performance is much less than the bound of $2+\Delta< 26.6$ guarantees.  $2+\Delta$ is a worse-case bound and while there are degenerate cases whose performance approaches this bound; it seems that in practice, the performance of the greedy policy is likely to be much better.

\begin{figure}[htb]
\begin{center}
\psfrag{pmaxpmin}[][][1.1]{$\Delta_{UB}$}
\psfrag{JoJg}[][][1.4][-90]{$\frac{J_t^*}{J_t^g}$}
\includegraphics[scale=.5]{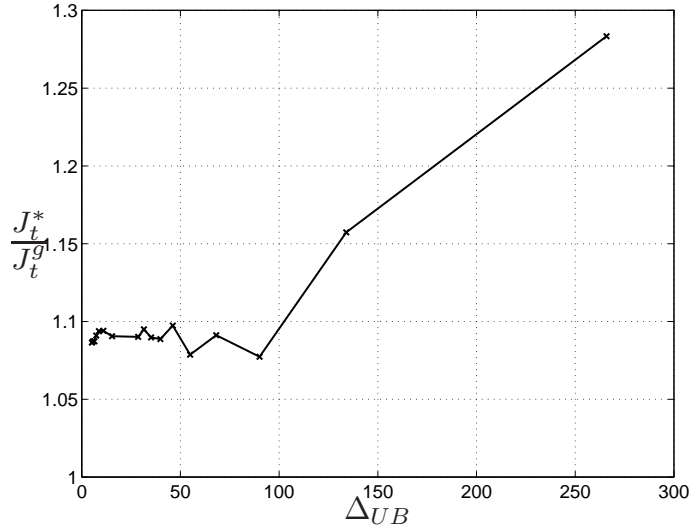}
 \caption{Performance loss ($\frac{J_t^*}{J_t^g}$) as $\Delta_{UB}$ increases.}\label{fig:runpmax}
\end{center}
\end{figure}

From Theorem \ref{theorem:bound}, the performance of the greedy policy is dependent on  $\Delta$, the ratio between the largest and smallest expected service times.  In our previous experiments, we have seen that $\frac{J_t^*}{J_t^g} \ll 2+\Delta$.  We now examine if the performance of the greedy policy will vary significantly as we change $\Delta$.  We fix $p_{\max} = .9$ and vary $p_{\min}\in [.01,.8]$; this varies the upper-bound of $\Delta$ (as calculated in Appendix \ref{ap:geo}), $\Delta_{UB} \in [5.2,265.9]$.  In Fig. \ref{fig:runpmax}, we see how the performance of the greedy policy ($\frac{J_t^*}{J_t^g}$) varies with $\Delta$.  As expected, as $\Delta$ increases, so does the loss in performance.  However, it is interesting to note that $\Delta$ must be very large before the degradation in performance is significant.   In fact, for a large range of $\Delta_{UB} \in [1,100]$, $\frac{J_t^*}{J_t^g}$ is nearly constant and the greedy policy performs within $10\%$ of optimal.  Even when $\Delta = 260$, $\frac{J_t^*}{J_t^g} < 1.3$. A loss of $30\%$ is \emph{much} better than the theory guarantees.

Depending on the system parameters, $\Delta$ can be arbitrarily large which would lead to arbitrarily large degradation in performance of the greedy policy. While we have seen via Example \ref{ex:suboptimality} that the performance of the greedy policy can be highly dependent on $\Delta$, we suspect this to be a degenerate example.  We expect that in practice, the performance of the greedy policy to be closer to performance of the optimal policy.

\section{Conclusion}\label{sec:conc}
In this paper, we have studied online stochastic non-preemptive scheduling of jobs with decaying rewards.  Arbitrary decaying reward functions enables this model to capture various distastes for delay which are more general than the standard exponential or linear decay as well as fixed (random or deterministic) deadlines.  Using stochastic Dynamic Programming techniques, we are able to show that a greedy heuristic is guaranteed to be within a factor of $\Delta + 2$ of optimal where $\Delta = \frac{E[\max_j \sigma_j]}{\min_j E[\sigma_j]}$ is the ratio of largest to shortest  service times.  While there exist degenerate scenarios where the performance loss of the proposed policy is near this worse-case bound, we expect that the performance loss to be much smaller for many practical scenarios of interest.

This is a first look at non-preemptive scheduling with arbitrary decaying rewards.  Some questions that remain are how to account for job arrivals and processor dependent service times.  When there are job arrivals, due to the non-preemptive service discipline, it may be optimal for a machine to idle in order to allow the machine to be free upon arrival of the new job.  However, doing so requires some estimate or knowledge of future jobs arrivals, which may not be available.  Also with processor dependent service times, optimal policies may also call for idling.  Consider a scenario where one machine is much faster than the rest.  Then an optimal policy may process \emph{all} jobs on this fast machine, causing the other machines to idle.  Allowing for idling policies significantly complicates the optimization problem at hand.  One option is to only consider non-idling policies and maximize reward over this class of policies.  It can be shown  via a highly degenerate example that requiring non-idling service disciplines can significantly degrade performance.  However, for many scenarios this constraint is very natural.  For instance, in service applications, such as health-care facilities, making customers (patients) wait when there are available servers (doctors) is unlikely to be tolerated.

These are just some extensions to this general model we have analyzed.  In this paper, we have considered the performance of an online scheduling algorithm for jobs with arbitrary decaying rewards.  We have shown a worse-case performance bound for this policy compared to the optimal \emph{off-line} algorithm.  While there are some rare instances when the loss in performance of the proposed greedy policy is significant, in practice, we expect the performance loss to be small.  This, along with the simplicity of this algorithm, makes it highly desirable for real world implementation.

\bibliographystyle{IEEEtran}
\bibliography{IEEEabrv,nonPR}
\appendix
\section{Bound on $\sigma_{\max}$}\label{ap:geo}
Suppose the service time of job $j$ is Geometrically distributed with probability $p_j$. Furthermore,  $p_j$ is uniformly distributed between $[p_{\min},p_{\max}]$.

Using the fact that $\sigma_j$ is Geometrically distributed, i.e. $P(\sigma_j \leq x) = 1 - (1-p_j)^x$ gives:
\begin{eqnarray}
P(\sigma_{\max} > x) &=& 1- \prod_{j=1}^J P(\sigma_j \leq x)\nonumber\\
&=& 1 - \prod_{j=1}^J \big (1 - (1-p_j)^x \big )\nonumber\\
&\leq & 1 -  \big (1 - (1-p_{\min})^x \big )^J
\end{eqnarray}
Finding the expectation of $\sigma_{\max}$ gives:
\begin{eqnarray}\label{eqn:sigmamax}
E[\sigma_{\max}] &=&\sum_{x=0}^\infty P(\sigma_{\max} >x) \nonumber \\
&\leq&\sum_{x=0}^\infty \Big [1 -  \big (1 - (1-p_{\min})^x \big )^J\Big ]
\end{eqnarray}

We can numerically solve (\ref{eqn:sigmamax}) to get an upper-bound on $E[\sigma_{\max}]$ and hence, an upper-bound on $\Delta$.  In particular:
$$\Delta \leq \Delta_{UB} =p_{\max} \sum_{x=0}^\infty \Big [1 -  \big (1 - (1-p_{\min})^x \big )^J\Big ]$$

\end{document}